\documentclass[11pt,reqno]{amsart}

\usepackage{amssymb}
\usepackage{amsmath}
\usepackage{amsthm}
\usepackage{latexsym}
\usepackage{amsfonts}
\usepackage{mathrsfs}
\usepackage{bbm}
\usepackage{titlesec}
\usepackage{color}
\usepackage{hyperref}
\usepackage{url}
\usepackage{xcolor} 

\newtheoremstyle{neu_thm}
{13pt}       
{8pt}      
{\itshape}  
{}          
{\bfseries} 
{.}         
{.5em}      
{}          

\newtheoremstyle{neu_defn}
{13pt}       
{8pt}      
{}  
{}          
{\bfseries} 
{.}         
{.5em}      
{}          

\newtheorem{thm}{Theorem}[section]
\newtheorem{cor}[thm]{Corollary}
\newtheorem{lem}[thm]{Lemma}
\newtheorem{prop}[thm]{Proposition}

\theoremstyle{neu_defn}
\newtheorem{defn}[thm]{Definition}
\newtheorem{rem}[thm]{Remark}

\newtheorem{prob}[thm]{Problem}

\titleformat{\section}{\normalfont\bfseries\centering}{\thesection.}{.25em}{}
\titleformat{\subsection}{\normalfont\bfseries}{\thesubsection.}{.25em}{}
\titlespacing{\section}{0pt}{*5}{*1.5}
\titlespacing{\subsection}{0pt}{*4}{*0.5}
\numberwithin{equation}{section}
\allowdisplaybreaks

\newcommand{\ol}{\overline}

\newcommand{\wt}{\widetilde}
\newcommand{\wh}{\widehat}

\renewcommand{\emptyset}{\varnothing}
\newcommand{\rmref}[1]{{\rm\ref{#1}}}
\newcommand{\diag}{\operatorname{diag}}

\newcommand{\braces}[1]{{\rm (}#1{\rm )}}

\newcommand{\R}{\ensuremath{\mathbb R}}    
\newcommand{\C}{\ensuremath{\mathbb C}}    
\newcommand{\Q}{\ensuremath{\mathbb Q}}    
\newcommand{\N}{\ensuremath{\mathbb N}}    
\newcommand{\Z}{\ensuremath{\mathbb Z}}    

\newcommand{\calD}{\mathcal D}         
         
\newcommand{\calF}{\mathcal F}         
\newcommand{\calG}{\mathcal G}

\newcommand{\calL}{\mathcal L}

\newcommand{\calS}{\mathcal S}

\newcommand{\la}{\lambda}
\newcommand{\veps}{\varepsilon}

\newcommand{\mat}[4]
{
   \begin{pmatrix}
      #1 & #2\\
      #3 & #4
   \end{pmatrix}
}

\newcommand{\smallmat}[4]{\left(\begin{smallmatrix}#1 & #2\\#3 & #4\end{smallmatrix}\right)}

\newcommand{\linspan}{\operatorname{span}}

\newcommand{\VMO}{\operatorname{VMO}}
\newcommand{\VMOli}{\operatorname{VMO}_{\rm loc}^\infty}
\newcommand{\BMO}{\operatorname{BMO}}
\newcommand{\fq}{\frac 1 {|Q|}}
\newcommand{\essinf}{\operatorname{ess\,inf}}
\newcommand{\sinc}{\operatorname{sinc}}
\newcommand{\bA}{{\mathbf A}}
\newcommand{\bR}{{\mathbf R}}
\newcommand{\bM}{{\mathbf M}}
\newcommand{\GL}{\operatorname{GL}}
\newcommand{\SL}{\operatorname{SL}}
\renewcommand{\H}{\mathbb H}

\begin{document}
\title[A Balian-Low Theorem for Subspaces]{A Balian-Low Theorem for Subspaces}

\author[A. Caragea]{Andrei Caragea}
\address{{\bf A.~Caragea:} KU Eichst\"att-Ingolstadt, Mathematisch-Geographische Fakult\"at, Ostenstra\ss e 26, Kollegiengeb\"aude I Bau B, 85072 Eichst\"att, Germany}
\email{andrei.caragea@gmail.com}
\urladdr{http://www.ku.de/?acaragea}

\author[D.G. Lee]{Dae Gwan Lee}
\address{{\bf D.G.~Lee:} KU Eichst\"att-Ingolstadt, Mathematisch-Geographische Fakult\"at, Ostenstra\ss e 26, Kollegiengeb\"aude I Bau B, 85072 Eichst\"att, Germany}
\email{daegwans@gmail.com}

\author[G.E. Pfander]{G\"otz E.~Pfander}
\address{{\bf G.E.~Pfander:} KU Eichst\"att-Ingolstadt, Mathematisch-Geographische Fakult\"at, Ostenstra\ss e 26, Kollegiengeb\"aude I Bau B, 85072 Eichst\"att, Germany}
\email{pfander@ku.de}
\urladdr{http://www.ku.de/?pfander}

\author[F. Philipp]{Friedrich Philipp}
\address{{\bf F.~Philipp:} KU Eichst\"att-Ingolstadt, Mathematisch-Geographische Fakult\"at, Ostenstra\ss e 26, Kollegiengeb\"aude I Bau B, 85072 Eichst\"att, Germany}
\email{fmphilipp@gmail.com}
\urladdr{http://www.ku.de/?fmphilipp}


\begin{abstract}
We extend the Balian-Low theorem to Gabor subspaces of $L^2(\R)$ by involving the concept of additional time-frequency shift invariance. We prove that if a Gabor system on a lattice of rational density is a Riesz sequence generating a subspace which is invariant under an additional time-frequency shift, then its generator cannot decay fast simultaneously in time and frequency.
\end{abstract}



\maketitle
\thispagestyle{empty}

%

\section[\hspace*{1cm}Introduction]{Introduction}

The Balian-Low theorem is an uncertainty principle in time-frequency analysis which in its original form states that a generator of a Gabor orthonormal basis of the space of square integrable functions on the real line cannot be well-localized simultaneously in time and frequency.

\begin{thm}[\cite{balian,low}]\label{t:balian-low}
If 
the functions $e^{2 \pi i n x} g(x-m)$, $(m,n) \in \Z\times\Z$, form an orthonormal basis of $L^2(\R)$, then
\begin{align}\label{eqn:uncertainty-product}
  \Big(\int |x-\alpha|^2 |g (x)|^2 \, dx \Big) \cdot \Big(\int |\omega-\beta|^2 |\widehat{g} (\omega)|^2 \, d\omega \Big) =\infty ,\quad \alpha,\beta\in\R .
\end{align}
\end{thm}

The result generalizes from $\Z \times \Z$ to separable lattices of the form $a \Z \times b \Z$, where $ab=1$; the latter being in fact necessary for $e^{2 \pi i bn x} g(x-a m)$, $(am,bn) \in a\Z\times b\Z$, to form an orthonormal basis of $L^2(\R)$. The results in this paper though achieve  generalizations in the case $a b >1$ by involving an additional invariance by time-frequency shifts. Before discussing the Balian Low theorem, its extensions, and our results in more depth, we state our main result in its simplest form  for illustration:

\begin{thm}\label{t:main-rational_density-simple}
If $ab\ge 1$ is rational, the functions $  e^{2 \pi i bn x} g(x-am)$,  $(am,bn) \in a\Z\times b\Z$, form an orthonormal system  and its closed linear span contains $e^{2 \pi i \eta x} g(x-u)$ for some $(u,\eta)\notin a\Z \times b\Z$, then
$$
\Big(\int |x-\alpha|^2 |g (x)|^2 \, dx \Big) \cdot \Big(\int |\omega-\beta|^2 |\widehat{g} (\omega)|^2 \, d\omega \Big) =\infty,\quad \alpha,\beta\in\R.
$$
\end{thm}


\vspace{.5cm}
In the last two decades, the Balian-Low theorem has inspired significant research in time-frequency analysis and has itself been generalized in various ways (see, e.g., \cite{asw,g,ht,no}). Gautam \cite{g} recognized that $g$ having a finite uncertainty product (\ref{eqn:uncertainty-product}) implies that its Zak transform $Zg$ has locally   vanishing mean oscillation   and that the latter actually prevents the system $\{e^{2\pi inx}g(x-m) : m,n\in\Z\}$ to be a Riesz basis of $L^2(\R)$. We will introduce the reader in Sections \ref{s:prelims} and \ref{s:VMO} to both the Zak transform and the concept of vanishing mean oscillation ($\VMO$). In fact, Gautam proved the following theorem.

\begin{thm}[\cite{g}]\label{t:gautam-ext}
If $g\in L^2(\R)$ such that the Gabor system $\{e^{2 \pi i n x}g(x-m) : m,n\in\Z\}$ is a Riesz basis of $L^2(\R)$, then $Zg\notin\VMO_{\rm loc}(\R^2)$. Moreover, if $Zg\notin\VMO_{\rm loc}(\R^2)$, then for any $p,q\in (1,\infty)$ with $\tfrac 1p + \tfrac 1q = 1$ we have
\begin{align}\label{eqn:uncertainty-product-pq}
\Big(\int |x-\alpha|^{q} |g (x)|^2 \, dx \Big) \cdot \Big(\int |\omega-\beta|^{p} |\widehat{g} (\omega)|^2 \, d\omega \Big) = \infty,\quad \alpha,\beta\in\R.
\end{align}
In particular, if $\{ e^{2 \pi i n x}g(x-m) : m,n\in\Z\}$ constitutes a Riesz basis of $L^2(\R)$, then \eqref{eqn:uncertainty-product-pq} holds for any $p$ and $q$ as above.
\end{thm}

In this paper, we generalize Theorem \ref{t:gautam-ext} in two ways. First, as the attentive reader might have noticed, Theorem \ref{t:gautam-ext} is only proved and formulated for the most simple lattice $\Z\times\Z$. In our results we consider general rational lattices and lattices of rational density. Secondly, we work with Gabor systems that constitute a Riesz basis of their closed linear span instead of $L^2(\R)$, as indicated in Theorem \ref{t:main-rational_density-simple}. Our first main result reads as follows.

\begin{thm}\label{t:main-rational_density}
Let $g\in L^2(\R)$ and let $\Lambda \subset \R^2$ be a lattice of rational density such that the Gabor system $\{e^{2\pi ibx}g(x-a) : (a,b)\in\Lambda\}$ is a Riesz basis of its closed linear span $\calG(g,\Lambda)$. If $e^{2\pi i\eta x}g(x-u)\in\calG(g,\Lambda)$ for some $(u,\eta)\notin\Lambda$, then \eqref{eqn:uncertainty-product-pq} holds for all $p,q\in (1,\infty)$ with $\tfrac 1 p+\tfrac 1 q = 1$.
\end{thm}

The conclusion of Theorem \ref{t:main-rational_density} can be strengthened significantly if we restrict ourselves to rational lattices, i.e., lattices that only consist of rational points. Recall that, given a field $\mathbb F$, by $\GL(n,\mathbb F)$ one usually denotes the group of invertible matrices in $\mathbb F^{n\times n}$ and $\SL(n,\mathbb F)$ stands for the subgroup of $\GL(n,\mathbb F)$ consisting of the matrices with determinant $1$.

\begin{thm}\label{t:main-rational}
Let $g\in L^2(\R)$ and let $\Lambda= A\Z^2$ with $A\in\operatorname{GL}(2,\Q)$, such that the Gabor system $\{e^{2\pi ibx}g(x-a) : (a,b)\in\Lambda\}$ is a Riesz basis of its closed linear span $\calG(g,\Lambda)$. If $e^{2\pi i\eta x}g(x-u)\in \calG(g,\Lambda)$ for some $(u,\eta)\notin\Lambda$, then $Zg\notin\VMO_{\rm loc}(\R^2)$.
\end{thm}

Note that neither of the above two theorems implies the other, since rational lattices are of rational density but the condition $Zg \notin \VMO_{\rm loc}(\R^2)$ is stronger than (\ref{eqn:uncertainty-product-pq}) as seen in Theorem \ref{t:gautam-ext} (cf.\ Problem \ref{prob:open}).

In \cite{gh}, Gebardo and Han already generalized the Balian-Low theorem to Gabor frames for subspaces of $L^2(\R)$. One of their main results states that {\it if $ab > 1$, $g\in L^2(\R)$, and $\{e^{2\pi ibmx}g(x-an) : m,n\in\Z\}$ forms an overcomplete frame for its closed linear span, then \eqref{eqn:uncertainty-product} holds}. As it is mentioned in \cite{gh}, the word ``overcomplete'' cannot be dropped in the statement, as is revealed by choosing the Gaussian for $g$ for which $\{e^{2\pi ibmx}g(x-an) : m,n\in\Z\}$ always is a Riesz sequence. Our theorems therefore complement the result from \cite{gh} inasmuch as we replace the term ``overcomplete frame for its closed linear span'' by ``Riesz basis of its closed linear span which has an additional time-frequency shift invariance''. In particular, we obtain the following corollary.

\begin{cor}
Let $g(x) = e^{-x^2}$ and let $\Lambda\subset\R^2$ be a lattice of rational density $<1$. Then $e^{2\pi i\eta x}g(x-u)\notin\calG(g,\Lambda)$ for all $(u,\eta)\in\R^2\backslash\Lambda$.
\end{cor}

An important variant of the Balian-Low theorem is the so-called amalgam version, known as the Amalgam Balian-Low theorem, which replaces the condition  \eqref{eqn:uncertainty-product} by $g\notin\calS_0 (\R)$, where $\calS_0 (\R)$ denotes the Feichtinger algebra, given by
$$
\calS_0 (\R)
= \left\{ f \in L^2(\R) : \int f(x) \, e^{-(x-t)^2}\ e^{2\pi i  x\nu}\, dx
\; \in L^{1} (t,\nu) \right\}.
$$
Recently \cite{cmp}, the Amalgam Balian-Low theorem has been generalized to Gabor subspaces of $L^2(\R)$ in a similar fashion as Theorem \ref{t:main-rational_density-simple} generalizes the Balian-Low theorem. Specifically, the main theorem in \cite{cmp} reads as Theorem \ref{t:main-rational_density} with ``\eqref{eqn:uncertainty-product-pq} holds for ...'' replaced by ``$g\notin\calS_0(\R)$''. In fact, the question whether $g\notin\calS_0(\R)$ can be replaced by \eqref{eqn:uncertainty-product} was posed as an open problem in \cite{cmp}. Hence, Theorem \ref{t:main-rational_density} gives a positive answer to this question and goes beyond.

As is well known, the techniques used in proving the Balian-Low theorem are much more involved than those used in the proof of the Amalgam Balian-Low theorem. Therefore, and as we want to point out, the problem of replacing $g\notin\calS_0(\R)$ by \eqref{eqn:uncertainty-product} or \eqref{eqn:uncertainty-product-pq} is by far not a matter of a straight-forward procedure.

The Balian-Low theorem and its amalgam version are not equivalent. In fact, as pointed out in \cite{bhw}, none of these two classical theorems implies the other. Therefore, it seems desirable to find a space $V\subset L^2(\R)$ which contains both $\calS_0(\R)$ and the set of functions with a finite uncertainty product as in \eqref{eqn:uncertainty-product} such that the functions in $V$ fail to be generators of Gabor Riesz bases of $L^2(\R)$. In fact, Theorem \ref{t:gautam-ext} provides such a space, namely the space of functions whose Zak transform is locally $\VMO$. Hence, the following easy consequence of Theorem \ref{t:main-rational} is a unification of the two classical theorems for rational lattices of the critical density $1$.

\begin{thm}\label{t:unification}
Let $g\in L^2(\R)$ and let $\Lambda= A\Z^2$ with $A\in\operatorname{GL}(2,\Q)$, $\det A = 1$, such that the Gabor system $\{e^{2\pi iax}g(x-b) : (a,b)\in\Lambda\}$ is a Riesz basis of $L^2(\R)$. Then $Zg\notin\VMO_{\rm loc}(\R^2)$.
\end{thm}

The paper is organized as follows. In Section \ref{s:prelims} we introduce the reader to the notions and notations used throughout the paper. Section \ref{s:VMO} introduces and discusses the functions that are locally of vanishing mean oscillation (VMO). We prove several statements on invariance properties of $\VMO_{\rm loc}(\R^n)$ which we make use of in the proofs of our main results Theorem \ref{t:main-rational_density} and Theorem \ref{t:main-rational} in Section \ref{s:proof}, but which also seem to be new and are interesting in their own right.

\section[\hspace*{1cm}Preliminaries]{Preliminaries}\label{s:prelims}
In this section, we collect basic notions and tools in time-frequency analysis that are necessary for formulating and proving our main results.
Recall that a lattice in $\R^2$ is a set of the form $A\Z^2$ with some $A\in\operatorname{GL}(2,\R)$ and its density is given by $|\det A|^{-1}$. We define the time-frequency shift operator by $(u,\eta)\in\R^2$ as
$$
\pi(u,\eta):L^2(\R) \rightarrow L^2(\R),
\quad
\pi(u,\eta)f (x) = e^{2\pi i\eta x}f(x-u) .
$$
Using this notation, the Gabor system generated by $g\in L^2(\R)$ and a lattice $\Lambda \subset \R^2$ is simply written as $(g,\Lambda) := \{ \pi(u,\eta) g  : (u,\eta) \in \Lambda \}$.
By $\calG(g,\Lambda)$ we denote its closed linear span in $L^2(\R)$, i.e., $\calG(g,\Lambda) = \ol{\linspan}\,\{ \pi(u,\eta) g  : (u,\eta) \in \Lambda \}$. For the convenience of the reader, we state some easily verifiable properties of the time-frequency shift operator in the following lemma.

\begin{lem}\label{l:pi}
The following statements hold.
\begin{enumerate}
\item[{\rm (a)}] For $a,b,c,d\in\R$, we have
$$
\pi(a,b)\pi(c,d) = e^{-2\pi iad}\pi(a+c,b+d) = e^{-2\pi i(ad-bc)}\pi(c,d)\pi(a,b).
$$
\item[{\rm (b)}] For fixed $f\in L^2(\R)$, the mapping $(a,b)\mapsto\pi(a,b)f$ is continuous from $\R^2$ to $L^2(\R)$.
\end{enumerate}
\end{lem}

The \emph{Fourier transform} is defined on $L^{1}(\R) \cap L^{2}(\R)$ by
\begin{align*}
\calF f(\omega) = \widehat{f}(\omega) := \int_{\R} f(x) \, e^{- 2 \pi i x \omega}\,dx,\quad \omega\in\R.
\end{align*}
It is well known that the operator $\mathcal{F}$ extends to a unitary operator from $L^{2}(\R)$ onto $L^{2}(\R)$. It can be used to define the Sobolev space $H^s(\R)$, $s>0$, as follows:
$$
H^s (\R) = \left\{ f \in L^2(\R) : \int_{\R}(1+ |\omega|^{2})^s|\widehat{f}(\omega)|^2\,d\omega < \infty\right\}.
$$
The \emph{Zak transform} of $f\in L^1(\R)\cap L^2(\R)$ is defined (a.e.) by
$$
Zf(x,\omega) = \sum_{k\in\Z}f(x+k)\,e^{-2\pi ik\omega},\quad (x,\omega)\in\R^2.
$$
As is easily seen, the function $Zf$ is quasi-periodic, i.e., for $m,n\in\Z$ we have
\begin{align}\label{e:extension}
Z f ( x + m , \omega + n ) = e^{2 \pi i m \omega} \, Z f (x, \omega)
\quad \text{for a.e. }(x,\omega)\in\R^2.
\end{align}
The mapping $f\mapsto Zf|_{[0,1]^2}$ extends continuously to a unitary map from $L^2(\R)$ onto $L^2([0,1]^2)$. Here, if $f\in L^2(\R)$, by $Zf$ we mean the quasi-periodic extension of $Zf|_{[0,1]^2}$ to $\R^2$, which is an element of $L^2_{\rm loc}(\R^2)$. We summarize some useful properties of the Zak transform in the following lemma.

\begin{lem}\label{l:zak}
Let $f\in L^2(\R)$. Then the following relations hold for a.e.\ $(x,\omega)\in\R^2$.
\begin{enumerate}
\item[{\rm (a)}] $Zf(x+m,\omega+n) = e^{2\pi im\omega} \, Zf(x,\omega)$ for all $m,n\in\Z$.
\item[{\rm (b)}] $Z\pi(u,\eta)f (x,\omega) = e^{2\pi i\eta x} \, Zf(x-u,\omega-\eta)$ for all $(u,\eta) \in \R^2$.
\item[{\rm (c)}] $(Z\pi(m,n)f)(x,\omega) = e^{2\pi i(nx-m\omega)} \, Zf(x,\omega)$ for all $m,n\in\Z$.
\item[{\rm (d)}] $Z\widehat{f} (x,\omega) = e^{2\pi i x \omega} \, Zf (-\omega,x)$.
\item[{\rm (e)}] $f(x) = \int_0^1 Zf(x,\omega)\,d\omega$.
\end{enumerate}
\end{lem}

The following technical lemma will be used to prove our main results. A similar statement can be found in \cite[Proposition 3.8]{gh} (see also \cite[Lemma 5]{cmp}). However, since the present setting is slightly different as in \cite{gh} and \cite{cmp}, we give a full proof of the statement.

\begin{lem}\label{l:riesz}
Let $P,Q\in\N$, $g\in L^2(\R)$, and assume that the Gabor system $(g,\tfrac 1 Q\Z\times P\Z)$ is a Riesz sequence in $L^2(\R)$ with Riesz bounds $A$ and $B$. Then the matrix function
$$
\bA(x,\omega) := \left( Zg(x-\tfrac k P - \tfrac \ell Q,\omega)\right)_{k,\ell=0}^{P-1,Q-1}\,\in\,\C^{P\times Q},\qquad (x,\omega)\in\R^2.
$$
is essentially bounded from above and from below. More precisely, for a.e.\ $(x,\omega)\in\R^2$ we have that
\begin{equation}\label{e:bounded}
PA\|\xi\|^2\,\le\,\|\bA(x,\omega)\xi\|^2\,\le\,PB\|\xi\|^2\qquad\text{for all $\xi\in\C^Q$}.
\end{equation}
In particular, $Zg\in L^\infty(\R^2)$, which also implies $g\in L^\infty(\R)$.
\end{lem}
\begin{proof}
Let $F\in L^\infty(R_P,\C^Q)$ be arbitrary, where $R_P := (0,\tfrac 1 P)\times (0,1)$. Then there exists $(c_{m,n})_{m,n\in\Z}\in\ell^2(\Z^2)$ such that $F_\ell = \sum_{s,n\in\Z}c_{sQ+\ell,n}e^{2\pi i(nPx-s\omega)}$, $\ell = 0,\ldots,Q-1$, where $F_\ell$ denotes the $\ell$-th coordinate of $F$. Using the properties of the Zak transform, we have (extending $F$ to $\R^2$ periodically)
\begin{align*}
\Bigg\|\sum_{m,n\in\Z}c_{m,n}\pi\left(\tfrac m Q,nP\right)g\Bigg\|_{L^2(\R)}^2
&= \Bigg\|\sum_{\ell=0}^{Q-1}\sum_{s,n\in\Z}c_{sQ+\ell,n}e^{2\pi i(nPx-s\omega)}(Zg)\left(x-\tfrac\ell Q,\omega\right)\Bigg\|_{L^2([0,1]^2)}^2\\
&= \left\|\sum_{\ell=0}^{Q-1}F_\ell(x,\omega)(Zg)\left(x-\tfrac\ell Q,\omega\right)\right\|_{L^2([0,1]^2)}^2\\
&= \int_0^{\tfrac 1 P}\int_0^1\sum_{k=0}^{P-1}\left|\sum_{\ell=0}^{Q-1}(Zg)\left(x-\tfrac k P-\tfrac\ell Q,\omega\right)F_\ell(x,\omega)\right|^2\!d\omega\,dx\\
&= \int_0^{\tfrac 1 P}\int_0^1\left\|\bA(x,\omega)F(x,\omega)\right\|_2^2\,d\omega\,dx.
\end{align*}
Hence, for every $F\in L^\infty(R_P,\C^Q)$ we obtain
$$
PA\|F\|_{L^2(R_P,\C^Q)}^2\,\le\,\int_0^{\tfrac 1 P}\int_0^1\left\|\bA(x,\omega)F(x,\omega)\right\|_2^2\,d\omega\,dx\,\le\,PB\|F\|_{L^2(R_P,\C^Q)}^2.
$$
Let $\calD$ be a countable dense set in $\C^Q$ (e.g., $\calD = (\Q+i\Q)^Q$). For $\xi\in\calD$ let $\calL(\xi)$ denote the set consisting of all Lebesgue points in $R_P$ of the map $(x,\omega)\mapsto\|\bA(x,\omega)\xi\|_2^2$ and put $\calL := \bigcap_{\xi\in\calD}\calL(\xi)$. Then $R_P\setminus\calL$ has zero measure. For $(x_0,\omega_0)\in\calL$, $\xi\in\calD$, and $\veps > 0$ define $F = F_{\xi,\veps,x_0,\omega_0} := \tfrac 1{\sqrt\pi\veps}\chi_{B_\veps(x_0,\omega_0)}\xi$, where $B_\veps(x_0,\omega_0)$ denotes the euclidian ball with center $(x_0,\omega_0)$ and radius $\veps$. Then, for $\veps$ small enough, $\|F\|_{L^2(R_P,\C^Q)} = \|\xi\|$ and
$$
\int_0^{\tfrac 1 P}\int_0^1\left\|\bA(x,\omega)F(x,\omega)\right\|_2^2\,d\omega\,dx = \tfrac 1{\pi\veps^2}\int_{B_\veps(x_0,\omega_0)}\left\|\bA(x,\omega)\xi\right\|_2^2\,d(x,\omega).
$$
Letting $\veps\to 0$ yields
$$
PA\|\xi\|^2\,\le\,\left\|\bA(x_0,\omega_0)\xi\right\|_2^2\,\le\,PB\|\xi\|^2.
$$
By a density argument this holds for all $\xi\in\C^Q$, which establishes \eqref{e:bounded} for all $(x,\omega)\in\calL$ and thus (due to the quasi-periodicity of $\bA$) for a.e.\ $(x,\omega)\in\R^2$. Now, choosing the first standard basis vector of $\C^Q$ for $\xi$, we obtain that $|Zg(x,\omega)|^2\le PB$ for a.e.\ $(x,\omega)\in\R^2$. Hence, $Zg\in L^\infty(\R^2)$. Also, Lemma \ref{l:zak}(e) yields $g\in L^\infty(\R)$.
\end{proof}

\section[\hspace*{1cm}Functions of Vanishing Mean Oscillation (VMO)]{Functions of Vanishing Mean Oscillation (VMO)}\label{s:VMO}
A cube in $\R^n$ of side length $\delta > 0$ is a set of the form $I_1\times\dots\times I_n$ where each $I_i \subset \R$ is a closed interval of length $\delta$. For a function $F\in L^1_{\rm loc}(\R^n)$ and a bounded measurable set $\Delta\subset\R^n$ with $|\Delta|>0$, we define
$$
F_\Delta := \frac 1{|\Delta|}\int_\Delta F\,dx
\qquad\text{and}\qquad
M_\Delta(F) := (|F - F_\Delta|)_\Delta.
$$
Also, for a bounded open set $U\subset\R^n$ and $\veps > 0$, let
$$
S_{\veps,U}(F) := \sup\left\{M_Q(F) : Q\subset U\text{ cube},\;|Q| < \veps\right\}.
$$

\begin{defn}
Let $U$ be a bounded open subset of $\R^n$.

\smallskip\noindent
(a) A function $F\in L^1_{\rm loc}(\R^n)$ is said to be of {\em bounded mean oscillation} \braces{$\BMO$} on $U$ if $\sup_Q M_Q(F) < \infty$, where the supremum is taken over all bounded cubes $Q$ contained in $U$.
The space of all such functions is denoted by $\BMO(U)$. We write $F\in\BMO_{\rm loc}(\R^n)$ if $F\in\BMO(U)$ for every bounded open set $U\subset\R^n$.

\smallskip\noindent
(b) A function $F\in L^1_{\rm loc}(\R^n)$ is said to be of {\em vanishing mean oscillation} \braces{$\VMO$} on $U$ if $F\in\BMO(U)$ and $\lim_{\veps\to 0}S_{\veps,U}(F) = 0$. The space of all such functions is denoted by $\VMO(U)$. Likewise, we write $F\in\VMO_{\rm loc}(\R^n)$ if $F\in\VMO(U)$ for every bounded open set $U\subset\R^n$.
\end{defn}

\begin{rem}
(a) It is easily seen that for any $F, G \in L^1_{\rm loc}(\R^n)$ and a cube $Q \subset U$, $(F+G)_Q = F_Q+G_Q$ and $M_Q(F+G)\leq M_Q(F)+M_Q(G)$, which leads to $S_{\veps,U}(F+G)\le S_{\veps,U}(F)+S_{\veps,U}(G)$. This shows that the sets $\BMO(U)$ and $\VMO(U)$ are linear spaces. Also, $\|\cdot\|_{\BMO(U)} := \sup_{Q\subset U}M_Q(F)$ induces a semi-norm on $\BMO(U)$.

\smallskip\noindent
(b) It is straightforward that $L^\infty(\R^n)\subset\BMO(\R^n)$. Also, every bounded uniformly continuous function on $\R^n$ belongs to $\VMO(\R^n)$ \cite{s}.
\end{rem}

In the sequel, we will use the notation
$$
\VMO_{\rm loc}^\infty(\R^n) := \VMO_{\rm loc}(\R^n)\cap L^\infty(\R^n).
$$
The following lemma shows in particular that $\VMO_{\rm loc}^\infty(\R^n)$ is closed under multiplication and is therefore an algebra.

\begin{lem}\label{l:closed}
The following statements hold.
\begin{itemize}
\item[{\rm (i)}]
If $F,G \in L^1_{\rm loc}(\R^n)$, then for any cube $Q\subset\R^n$ we have
\begin{equation}\label{e:FQs}
|F_QG_Q - (FG)_Q|\le \tfrac{1}{2}\max\{\|F\|_\infty,\|G\|_\infty\}(M_Q(F) + M_Q(G)).
\end{equation}
Also, whenever $U\subset\R^n$ is a bounded open set and $\veps > 0$, then
\begin{equation}\label{e:Se}
S_{\veps,U}(FG)\,\le\, \tfrac{3}{2} \max\{\|F\|_\infty,\|G\|_\infty\}\big(S_{\veps,U}(F)+S_{\veps,U}(G)\big).
\end{equation}
Consequently, $F,G\in\VMOli(\R^n)$ implies $FG\in\VMOli(\R^n)$.

\item[{\rm (ii)}]
If $F\in\VMO_{\rm loc}(\R^n)$ and $\essinf |F| > 0$, then for any bounded open set $U\subset\R^n$ there exists $\veps = \veps_U > 0$ such that
\begin{align}\label{e:FQ_greater}
|F_Q|\geq (\essinf |F|)/2
\quad \text{for all cubes} \;\; Q\subset U \;\; \text{with} \;\; |Q|<\veps ,
\end{align}
and
\begin{equation}\label{e:Se2}
S_{\veps,U}(1/F)\,\le\,\frac 4{(\essinf|F|)^2} \, S_{\veps,U}(F).
\end{equation}
Consequently, $F\in\VMO_{\rm loc}(\R^n)$ and $\essinf |F| > 0$ imply $1/F\in\VMOli(\R^n)$.
\end{itemize}
\end{lem}
\begin{proof}
(i) Let $F,G \in L^1_{\rm loc}(\R^n)$ and let $Q\subset U$ be a cube. Then
\begin{align*}
\fq\int_Q|FG-(FG)_Q|\,dx
&\le\fq\int_Q|FG-F_QG_Q|\,dx + |F_QG_Q - (FG)_Q|.
\end{align*}
We estimate the first term on the right hand side as
\begin{align*}
\fq\int_Q|FG-F_QG_Q|\,dx
&\le\fq\int_Q\left(|F||G-G_Q| + |F-F_Q||G_Q|\right)\,dx\\
&\le\max\{\|F\|_\infty,\|G\|_\infty\}(M_Q(F) + M_Q(G)).
\end{align*}
For the second term, we observe that
\begin{align*}
|F_QG_Q - (FG)_Q|
&= \frac 1 {|Q|}\left|\int_Q F\left(G_Q - G\right)\,dx\right|
\le\frac 1{|Q|}\int_Q|F|\left|G - G_Q\right|\,dx\,\le\,\|F\|_\infty M_Q(G)
\end{align*}
and
\begin{align*}
|F_QG_Q - (FG)_Q|
&= \frac 1 {|Q|}\left|\int_Q \left(F_Q - F\right) G \,dx\right|
\le\frac 1{|Q|}\int_Q\left|F - F_Q\right| |G|\,dx\,\le\,\|G\|_\infty M_Q(F)
\end{align*}
so that
\begin{align*}
|F_QG_Q - (FG)_Q|
&\le \tfrac{1}{2} \max\{\|F\|_\infty,\|G\|_\infty\}(M_Q(F) + M_Q(G)).
\end{align*}
Therefore, $M_Q(FG)\le \frac{3}{2} \max\{\|F\|_\infty,\|G\|_\infty\}(M_Q(F)+M_Q(G))$ from which \eqref{e:Se} follows.

(ii) Assume that $F\in\VMO_{\rm loc}(\R^n)$ and $C := \essinf|F| > 0$, and let $U\subset\R^n$ be an open set. Since $F\in\VMO_{\rm loc}(\R^n)$, we have $S_{\veps,U}(F)\le C/2$ for some  $\veps = \veps_U > 0$.
Let $Q\subset U$ be any cube with $|Q|<\veps$. Then $|F(x) - F_Q| + |F_Q| \ge |F(x)| \geq C$ a.e.~so that $M_Q(F) + |F_Q| \ge C$. Using the fact that $M_Q(F) \leq S_{\veps,U}(F)\le C/2$, we obtain $|F_Q|\geq C/2$. Now, observe that
$$
M_Q(1/F) = \fq\int_Q\left|\frac 1 F - \left(\frac 1 F\right)_Q\right|dx
\,\le\,\fq\int_Q\left|\frac 1 F - \frac 1 {F_Q}\right|dx + \left|\frac 1 {F_Q} - \left(\frac 1 F\right)_Q \right|.
$$
The first term can be estimated by
$$
\fq\int_Q\left|\frac 1 F - \frac 1 {F_Q}\right|dx = \fq\int_Q\frac{|F_Q-F|}{|F_QF|}\,dx\,\le\,\frac{2M_Q(F)}{C^2}
$$
and the second term by
\begin{align*}
\left| \frac 1 {F_Q} - \left(\frac 1 F\right)_Q \right|
&\le\frac 2 C\left| 1 - F_Q\left(\frac 1 F\right)_Q \right| = \frac 2{C|Q|}\left|\int_Q\left(1 - \frac{F_Q}{F}\right)dx\right|\\
&\le\frac 2{C^2|Q|}\int_Q\left|F - F_Q\right|\,dx = \frac{2M_Q(F)}{C^2}.
\end{align*}
Thus, we have $M_Q(1/F) \le 4M_Q(F)/C^2$, which yields \eqref{e:Se2} .
\end{proof}

A successive application of \eqref{e:FQs} in Lemma \ref{l:closed}(i) gives the following corollary.

\begin{cor}\label{c:prods}
If $F,F_1,\ldots,F_n\in\VMOli(\R^n)$, then $\prod_{i=1}^nF_i\in\VMOli(\R^n)$. Moreover, there exists a constant $C>0$ which only depends on $\|F_1\|_\infty,\ldots,\|F_n\|_\infty$ such that for any bounded open set $U\subset\R^n$, $\veps>0$ and a cube $Q\subset U$ with $|Q|<\veps$,
$$
\left|\left(\prod_{i=1}^nF_i\right)_{\!\!Q} - \prod_{i=1}^n(F_i)_Q\right|\,\le\,C\sum_{i=1}^nS_{\veps,U}(F_i) .
$$
\end{cor}

\begin{cor}\label{c:matrix}
Let ${\mathbf B} (x) = [B_{j,k} (x)]_{j,k=1}^N$, $N\in\N$, be such that each $B_{j,k} : \R^n \rightarrow \C$ belongs to $\VMOli(\R^n)$ and $\overline{B_{k,j}} (x) = B_{j,k} (x)$ for all $j,k$, i.e., ${\mathbf B}(x) = {\mathbf B}(x)^*$. If there exist constants $\alpha,\beta > 0$ such that $\alpha I_N\le {\mathbf B}(x)\le\beta I_N$ a.e., then each entry of ${\mathbf B}(x)^{-1}$ belongs to $\VMOli(\R^n)$.
\end{cor}
\begin{proof}
Let ${\mathbf C}(x) = \operatorname{adj}{\mathbf B}(x)$ be the adjugate matrix of ${\mathbf B}(x)$. By Lemma \ref{l:closed}(i), each entry of ${\mathbf C}(x)$ is in $\VMOli(\R^n)$ and so is $\det {\mathbf B} (x)$. Since $\det {\mathbf B}(x)\ge\alpha^N$ a.e., it follows from Lemma \ref{l:closed}(ii) that $(\det {\mathbf B}(x))^{-1}\in\VMOli(\R^n)$. Again by Lemma \ref{l:closed}(i), we conclude that each entry of ${\mathbf B}(x)^{-1} = (\det {\mathbf B}(x))^{-1}{\mathbf C}(x)$ belongs to $\VMOli(\R^n)$.
\end{proof}

The next proposition will play a key role in the proofs of our main theorems. It was proved for a continuous function $H$ in \cite[Proposition 3]{cmp}. Here, we relax the condition to $H\in\VMO_{\rm loc}^\infty(\R^2)$ which is much weaker than $H$ being continuous.

\begin{prop}\label{p:THEPROP}
Let $P_1,P_2,N\in\N$, $M_1,M_2 \in\Z$, and $(u,\eta)\in\Q^2$, $(u,\eta)\neq (0,0)$, such that $Nu,N\eta\in\Z$. If $H\in\VMOli(\R^2)$ is $\frac{1}{P_1}$-periodic in $x$, $\frac{1}{P_2}$-periodic in $\omega$ and
\begin{equation}\label{e:product}
\prod_{n=0}^{N-1} H(x+nu,\omega + n\eta) = e^{2\pi i(M_1x + M_2\omega)}\quad\text{for a.e.\ $(x,\omega)\in\R^2$},
\end{equation}
then $NP_1$ divides $M_1$ and $NP_2$ divides $M_2$.
\end{prop}
\begin{proof}
First, we note that $H\in\VMO(\R^2)$ since $H$ is periodic. For $r > 0$ and $F\in L^\infty(\R^2)$, we define the mean function $$
F_{[r]}(x,\omega) :=
\frac 1{|Q_r(x,\omega)|}\int_{Q_r(x,\omega)}F(z)\,dz,\qquad (x,\omega) \in\R^2,
$$
which takes the average of $F$ over the cube $Q_r(x,\omega)$ of side length $r$ centered at $(x,\omega)$,
$$
Q_r(x,\omega) = [x-\tfrac r 2,x+\tfrac r 2]\times [\omega-\tfrac r 2,\omega+\tfrac r 2].
$$
It is easily seen that $F_{[r]}$ is continuous (even Lipschitz continuous); moreover, if $F$ is periodic, then $F_{[r]}$ inherits the periodicity of $F$. Setting $H_n(x,\omega) := H(x+nu,\omega+n\eta)$ for $n=0,\ldots,N-1$, Corollary \ref{c:prods} implies that
\begin{align*}
\left| \prod_{n=0}^{N-1}(H_n)_{[r]} - \left(\prod_{n=0}^{N-1} H_n \right)_{[r]} \right|
&\le C\sum_{n=0}^{N-1}S_{r^{2},\R^2}(H_n) = CNS_{r^{2},\R^2}(H),
\end{align*}
where $C > 0$ is a constant which depends only on $\|H\|_\infty$. Using
$$
\left(\prod_{n=0}^{N-1}H_n \right)_{[r]} (x,\omega)
= \left( e^{2\pi i(M_1x+M_2\omega)} \right)_{[r]}
= \sinc(M_1r)\sinc(M_2r) \, e^{2\pi i(M_1x+M_2\omega)},
$$
we obtain
\begin{align*}
\left| \prod_{n=0}^{N-1} (H_n)_{[r]}(x,\omega) - e^{2\pi i(M_1x+M_2\omega)} \right| 
&\le  CNS_{r^{2},\R^2}(H) + |\sinc(M_1r)\sinc(M_2r) -1|
\end{align*}
for all $(x,\omega)\in\R^2$. Note that the right hand side does not depend on $(x,\omega)$ and tends to zero as $r\to 0$. Since $\prod_{n=0}^{N-1} (H_n)_{[r]}(x,\omega)$ and $e^{2\pi i(M_1x+M_2\omega)}$ are continuous functions in $(x,\omega)$, there exist continuous functions $\rho_r : \R^2 \rightarrow \C$, $r > 0$, such that
\begin{align*}
\prod_{n=0}^{N-1} (H_n)_{[r]}(x,\omega)
&= \rho_r(x,\omega) \, e^{2\pi i(M_1x+M_2\omega)} .
\end{align*}
It is easily seen that $\rho_r(x,\omega)$ converges uniformly to $1$ on $\R^2$ as $r \rightarrow 0$. Noting that $(H_n)_{[r]}(x,\omega) = H_{[r]}(x+nu,\omega+n\eta)$ for $n=0,\ldots,N-1$, the equation above can be written as
\begin{align}\label{e:product_rho_r}
\prod_{n=0}^{N-1} H_{[r]}(x+nu,\omega+n\eta)
&= \rho_r(x,\omega) \, e^{2\pi i(M_1x+M_2\omega)} .
\end{align}
Here, the mean function $H_{[r]}$ inherits the periodicity of $H$, and is therefore $\frac{1}{P_1}$-periodic in $x$ and $\frac{1}{P_2}$-periodic in $\omega$.
Note that the periodicity of $H$ together with (\ref{e:product}) yields $M_1/P_1,M_2/P_2\in\Z$. This shows that $e^{2\pi i(M_1x+M_2\omega)}$ is $\frac{1}{P_1}$-periodic in $x$ and $\frac{1}{P_2}$-periodic in $\omega$, and therefore by (\ref{e:product_rho_r}), so is $\rho_r(x,\omega)$.
On the other hand, by replacing $x$ and $\omega$ respectively with $x+u$ and $\omega+\eta$ in (\ref{e:product}), taking into account $Nu,N\eta\in\Z$, and using the periodicity of $H$, we find that $M_1u+M_2\eta\in\Z$. Applying the same trick to (\ref{e:product_rho_r}) then gives
$$
\rho_r(x+u,\omega+\eta) = \rho_r(x,\omega)
$$
for all $r > 0$ and $(x,\omega) \in \R^2$. As $\rho_r\to 1$ uniformly, there exists a branch of $\sqrt[N]{\,\cdot\,}$ such that $\sqrt[N]{\rho_r}$ is continuous for $r$ small enough, say, $r\le r_0$, $r_0 > 0$. Now, setting $G_r(x,\omega) := H_{[r]}(x,\omega)/\sqrt[N]{\rho_r(x,\omega)}$ for $r\le r_0$ and combining all these facts yields
\begin{align*}
\prod_{n=0}^{N-1} G_r(x+nu,\omega + n\eta)
&= \prod_{n=0}^{N-1} \frac{H_{[r]}(x+nu,\omega + n\eta)}{\sqrt[N]{\rho_r(x+nu,\omega + n\eta)}}
= e^{2\pi i(M_1x+M_2\omega)}.
\end{align*}
Note that $G_r$ is continuous and $\frac{1}{P_1}$-periodic in $x$, $\frac{1}{P_2}$-periodic in $\omega$. The fact that $NP_1$ divides $M_1$ and $NP_2$ divides $M_2$ now follows from \cite[Proposition 3]{cmp}.
\end{proof}

In the remainder of this section, we consider functions in $\VMO_{\rm loc}(\R^n)$ that are not necessarily in $L^\infty(\R^n)$.

\begin{lem}\label{l:first}
Let $\Delta_1,\Delta_2\subset\R^n$ be bounded measurable sets with $\Delta_1\subset\Delta_2$ and $|\Delta_1| > 0$. Then for any $F\in L^1_{\rm loc}(\R^n)$ we have
$$
M_{\Delta_1}(F)\,\le\,2\frac{|\Delta_2|}{|\Delta_1|}\,M_{\Delta_2}(F).
$$
\end{lem}
\begin{proof}
Note that for any $F\in L^1_{\rm loc}(\R^n)$,
\begin{align*}
M_{\Delta_2}(F)
&= \frac 1{|\Delta_2|}\int_{\Delta_2}|F - F_{\Delta_2}|\,dx
\,\geq\,\frac 1{|\Delta_2|}\int_{\Delta_1}|(F - F_{\Delta_1}) - (F_{\Delta_2}-F_{\Delta_1})|\,dx\\
&\geq \frac{|\Delta_1|}{|\Delta_2|}\big(M_{\Delta_1}(F) - |F_{\Delta_2} - F_{\Delta_1}|\big),
\end{align*}
which is equivalent to
$$
M_{\Delta_1}(F)\le\frac{|\Delta_2|}{|\Delta_1|}M_{\Delta_2}(F) + |F_{\Delta_2} - F_{\Delta_1}|.
$$
Estimating the last term by
\begin{align*}
|F_{\Delta_2} - F_{\Delta_1}|
&= \left|\frac 1{|\Delta_1|}\int_{\Delta_1} (F - F_{\Delta_2})\,dx\right|\,\le\,\frac 1{|\Delta_1|}\int_{\Delta_1}\left|F - F_{\Delta_2}\right|\,dx
\le \frac{|\Delta_2|}{|\Delta_1|}M_{\Delta_2}(F) ,
\end{align*}
we obtain the desired inequality.
\end{proof}

\begin{lem}\label{l:affine}
Let $A\in\operatorname{GL}(n,\R)$ and $b\in\R^n$ and define an affine mapping $\Phi : \R^n\to\R^n$ by $\Phi(x) = Ax + b$. Then for any $F\in L^1_{\rm loc}(\R^n)$ and any cube $Q\subset\R^n$ with center $c\in\R^n$ and side length $\delta > 0$ we have
$$
M_Q(F\circ\Phi)\,\le\,\frac{2n^{n/2}\|A\|_{\rm op}^n}{|\det A|}\,M_{\wt Q}(F),
$$
where $\wt Q$ is the cube with center $\Phi(c)$ and side length $\sqrt n\|A\|_{\rm op}\delta$. Consequently, if $F\in\VMO_{\rm loc}(\R^n)$, then $F\circ\Phi\in\VMO_{\rm loc}(\R^n)$.
\end{lem}
\begin{proof}
Note that the set $\Phi(Q)$ is a parallelepiped in $\R^n$ with volume $|AQ| = |\det A||Q|$. For any $G\in L^1_{\rm loc}(\R^n)$, we have
$$
(G\circ\Phi)_Q = \frac 1{|\det A||Q|}\int_{\Phi^{-1}(\Phi(Q))}G(\Phi(x))|\det A|\,dx = \frac 1{|\Phi(Q)|}\int_{\Phi(Q)}G\,dx = G_{\Phi(Q)}
$$
so that
$$
M_Q(F\circ\Phi) = (|F\circ\Phi - (F\circ\Phi)_Q|)_Q = (|F - F_{\Phi(Q)}|\circ\Phi)_Q = (|F-F_{\Phi(Q)}|)_{\Phi(Q)} = M_{\Phi(Q)}(F).
$$
It is easy to see that the cube $\wt Q$ contains $\Phi(Q)$. Hence, Lemma \ref{l:first} implies that
$$
M_Q(F\circ A)\,\le\,2\frac{|\wt Q|}{|\Phi(Q)|}M_{\wt Q}(F) = \frac{2n^{n/2}\|A\|_{\rm op}^n}{|\det A|}\,M_{\wt Q}(F).
$$
This proves the lemma.
\end{proof}

\begin{prop}\label{prop:puh}
For $F\in L^n_{\rm loc}(\R^n)$, $\phi\in C^1(\R^n)$ and a cube $Q\subset\R^n$ of side length $\delta > 0$, we have
$$
M_Q(\phi F)\,\le\,
\|\phi\|_{L^\infty(Q)}M_Q(F) + \|\phi'\|_Q\|F\|_{L^n(Q)} ,
$$
where $\|\phi'\|_Q = \sup_{x\in Q}\|\nabla\phi(x)\|_{\ell_1}$.
Consequently, if $F\in\VMO_{\rm loc}(\R^n) \cap L^n_{\rm loc}(\R^n)$ and $\phi\in C^1(\R^n)$, then $\phi F\in\VMO_{\rm loc}(\R^n)$.
\end{prop}
\begin{proof}
Let $F\in L^n_{\rm loc}(\R^n)$, $\phi\in C^1(\R^n)$ and $Q\subset\R^n$ a cube of side length $\delta > 0$. Then for any $x,y\in Q$,
$$
|\phi(x)-\phi(y)|\,\le\,\|\phi'\|_Q\cdot\|x-y\|_\infty\,\le\,\delta\|\phi'\|_Q
$$
so that
\begin{align*}
M_Q(\phi F)
&\le\frac 1{|Q|}\int_Q|\phi(F-F_Q)|\,dx + \frac 1{|Q|}\int_Q|\phi F_Q - (\phi F)_Q|\,dx\\
&\le \|\phi\|_{L^\infty(Q)}M_Q(F) + \frac 1{|Q|}\int_Q\left|\frac 1{|Q|}\int_Q [\phi(x) - \phi(y)] \, F(y)\,dy\right|\,dx\\
&\le \|\phi\|_{L^\infty(Q)}M_Q(F) + \frac{\delta\|\phi'\|_Q}{|Q|}\int_Q|F(y)|\,dy\\
&\le \|\phi\|_{L^\infty(Q)}M_Q(F) + \frac{\delta\|\phi'\|_Q}{|Q|}\|F\|_{L^n(Q)}|Q|^{1 - 1/n} ,
\end{align*}
which gives the desired estimate.
\end{proof}

As for any $f\in L^2(\R)$ the Zak transform $Zf$ is locally square-integrable, we deduce the following corollary.

\begin{cor}\label{c:puh}
If $f\in L^2(\R)$ satisfies $Zf\in\VMO_{\rm loc}(\R^2)$, then for any $\phi\in C^1(\R^2)$ we have that $\phi Zf\in\VMO_{\rm loc}(\R^2)$.
\end{cor}

\begin{prop}\label{prop:dilations-chirps}
For $\alpha,\beta\in \R \backslash \{0\}$, define the operators
\begin{align*}
& D_\alpha : L^2(\R) \rightarrow L^2(\R), \quad
D_\alpha f(x) = \sqrt{|\alpha|}f(\alpha x)
\quad (\textnormal{dilation by} \; \alpha) \\
& C_\beta : L^2(\R) \rightarrow L^2(\R), \quad
C_\beta f(x) = e^{2\pi i\beta x^2} f(x)
\quad (\textnormal{multiplication by a chirp} \; e^{2\pi i\beta x^2} ) .
\end{align*}
If $g \in L^2(\R)$ satisfies $Zg \in\VMO_{\rm loc}(\R^2)$, then $Z\hat{g},\,Z(D_\alpha g),\,Z(C_\beta g)\in\VMO_{\rm loc}(\R^2)$ for all $\alpha,\beta\in \Q \backslash \{0\}$.
\end{prop}

\begin{proof}
By Lemma \ref{l:zak}, we have
\begin{align*}
Z \hat{g} (x,\omega)
&= e^{2\pi i x \omega} Zg (-\omega,x)
= e^{2\pi i x \omega} Zg \left( {\smallmat 0 {-1} 1 0} (x,\omega)^T \right) .
\end{align*}
Then Lemma \ref{l:affine} and Corollary \ref{c:puh} immediately yield that $Z \hat{g} \in\VMO_{\rm loc}(\R^2)$. Next, let us write $\alpha = p/q \in \Q \backslash \{0\}$, where $p,q\in\Z\setminus\{0\}$ are coprime, and $A = \diag(\alpha,\alpha^{-1})$. It is known \cite{j} that
$$
Z(D_\alpha g)(x,\omega) = \frac 1{\sqrt{pq}}\sum_{\ell=0}^{q-1}\sum_{r=0}^{p-1}e^{2\pi i\ell\omega}Zg\left(A(x,\omega)^T + (\alpha\ell,r/p)^T \right)
\quad \text{if} \;\; \alpha > 0
$$
and
$$
Z(D_\alpha g)(x,\omega) = \frac 1{\sqrt{-pq}}\sum_{\ell=0}^{q-1}\sum_{r=0}^{p-1}e^{-2\pi i\ell\omega}Zg\left(A(x,\omega)^T - (\alpha\ell,r/p)^T \right)
\quad \text{if} \;\; \alpha < 0 .
$$
Likewise, Lemma \ref{l:affine} and Corollary \ref{c:puh} imply that $Z(D_\alpha g)\in\VMO_{\rm loc}(\R^2)$. Finally, observe that
\begin{equation}\label{e:easy_chirp}
C_\beta = D_{\gamma^{-1}}C_{\beta \gamma^2}D_{\gamma}
\end{equation}
for any $\beta,\gamma\in\R\backslash\{0\}$. If $\beta\in \Q \backslash \{0\}$, then there exists $\gamma \in\N$ such that $\beta \gamma^2 \in\Z$. Therefore, it suffices to show that $Z(C_m g)\in\VMO_{\rm loc}(\R^2)$ for $m\in\Z$. For this, note that
\begin{align*}
Z(C_mg)(x,\omega)
&= \sum_{k\in\Z}e^{2\pi im(x+k)^2}g(x+k)e^{-2\pi ik\omega}
= e^{2\pi imx^2}\sum_{k\in\Z}g(x+k)e^{-2\pi ik(\omega-2mx)} \\
&= e^{2\pi imx^2} Zg(x,\omega-2mx)
= e^{2\pi imx^2} Zg \left( {\smallmat 1 0 {-2m} 1} (x,\omega)^T \right) .
\end{align*}
Again, the claim now follows from Lemma \ref{l:affine} and Corollary \ref{c:puh}.
\end{proof}

\begin{rem}\label{r:notVMO}
In \cite{g}, it is claimed, referring to the quasi-periodicity of Zak transform, that $Zf\in\VMOli(\R^2)$ implies $Zf\in\VMO(\R^2)$.
However, this is not true in general. For example, the function $f(x) = {\bf 1}_{[0,1)}(x)\sin(\pi x)$ satisfies $Zf\in\VMOli(\R^2) \backslash \VMO(\R^2)$. To see this, given any $\delta\in (0,1)$ let $k\in\N$ be such that $|\sinc(k\delta)|\le\tfrac 1 2$ and
$$
Q = [k+\tfrac {1-\delta} 2,k+\tfrac {1+\delta} 2] \times [-\tfrac\delta 2,\tfrac\delta 2] .
$$
Then
\begin{align*}
(Zf)_Q
&= \left(\frac 1\delta\int_{k+(1-\delta)/2}^{k+(1+\delta)/2}\sin(\pi(x-k))\,dx\right)\left(\frac 1\delta\int_{-\delta/2}^{\delta/2}e^{2\pi ik\omega}\,d\omega\right)=\sinc(\delta/2)\sinc(k\delta)
\end{align*}
and thus
\begin{align*}
M_Q(Zf)
&= \frac 1 {\delta^2}\int_Q\left|\sin(\pi(x-k))e^{2\pi ik\omega}-\sinc(\delta/2)\sinc(k\delta)\right|\,d(x,\omega)\\
&\ge \frac 1 {\delta^2}\int_Q\big[\sin(\pi(x-k)) - \sinc(\delta/2)|\sinc(k\delta)|\big]\,d(x,\omega)\\
&= \sinc(\delta/2)(1-|\sinc(k\delta)|)\,\ge\,\sinc(1/2)(1-|\sinc(k\delta)|)\,\ge\,\frac 1\pi,
\end{align*}
which shows that $Zf\notin\VMO(\R^2)$.
\end{rem}

\section[\hspace*{1cm}Proof of Theorems \ref{t:main-rational_density} and \ref{t:main-rational}]{Proof of Theorems \ref{t:main-rational_density} and \ref{t:main-rational}}\label{s:proof}
The core of the proof of both theorems is the following proposition, which is simply Theorem \ref{t:main-rational} for lattices of the form $\tfrac 1Q\Z\times P\Z$. In the proofs of the two theorems below we shall extend Proposition \ref{p:separable} to more general lattices by means of so-called metaplectic operators.

\begin{prop}\label{p:separable}
Let $g\in L^2(\R)$ and let $\Lambda = \tfrac 1 Q\Z\times P\Z$ with $P,Q\in\N$ coprime, such that the Gabor system $\{e^{2\pi ibx}g(x-a) : (a,b)\in\Lambda\}$ is a Riesz basis of its closed linear span $\calG(g,\Lambda)$. If $e^{2\pi i\eta x}g(x-u)\in \calG(g,\Lambda)$ for some $(u,\eta)\in\R^2\backslash\Lambda$, then $Zg\notin\VMO_{\rm loc}(\R^2)$.
\end{prop}

Before we prove Proposition \ref{p:separable}, we state a lemma that allows us to replace $(u,\eta)\in \R^2 \backslash \Lambda$ with a rational pair $(u,\eta)\in\Q^2\backslash\Lambda$.

\begin{lem}\label{l:rational}
For $g\in L^2(\R)$ and $\Lambda = \frac{1}{Q}\Z\times P\Z$ with $P,Q\in\N$, define the set
$$
M = \left\{(\alpha,\beta)\in\R^2 : \pi(\alpha,\beta)g\in\calG(g,\Lambda)\right\}
\quad (\supset \Lambda) .
$$
If $M\backslash\Lambda\neq\emptyset$, then there exists $(u,\eta)\in (M\cap\Q^2)\backslash\Lambda$.
\end{lem}
\begin{proof}
Note that, by Lemma \ref{l:pi}, $M$ is a closed set containing the lattice $\Lambda = \frac{1}{Q}\Z\times P\Z$. Assume that $(u,\eta)\in M\backslash\Lambda$ and define the set
$$
\Lambda_1 := \left\{(\tfrac m Q,nP) + k(u,\eta) : m,n,k\in\Z\right\}\,\subset\, M.
$$
Assume that $u\notin\Q$, $\eta\in\Q$, $\eta = \tfrac c d$, $c,d\in\Z$, $d\neq 0$, and put $v := dPu\notin\Q$. Let $x:=\tfrac 1{2Q}$. Since $\{\tfrac m Q + \ell v : m,\ell \in\Z\}$ is dense in $\R$, we can choose $m,\ell \in\Z$ such that $x' := u + \tfrac m Q + \ell v$ is arbitrary close to $x$. Then
$$
(x',\eta) = \left(\tfrac m Q,-\ell cP\right) + (1+\ell dP)(u,\eta)\,\in\,\Lambda_1.
$$
Hence, $(\tfrac 1{2Q},\eta)\in\ol{\Lambda_1}\backslash\Lambda\subset M\backslash\Lambda$.

Assume now that $\eta\notin\Q$. Since $\{nP + k\eta : n,k\in\Z\}$ is dense in $\R$, there exist sequences $(n_j)_{j\in\N}$ and $(k_j)_{j\in\N}$ in $\Z$ such that $n_jP + k_j\eta\to\tfrac 1{2}$ as $j\to\infty$. For each $j\in\N$ pick $m_j\in\Z$ such that $\tfrac{m_j}Q + k_ju\in [0,\tfrac 1Q]$. Then this sequence is bounded and thus has a convergent subsequence. By $x$ denote its limit. Then $(x,\tfrac 12)\in M\backslash\Lambda$. If $x\in\Q$, we have reached our aim. Otherwise the above reasoning applies again.
\end{proof}

\begin{proof}[Proof of Proposition \rmref{p:separable}]
In the sequel we aim to deduce a contradiction from the three following assumptions (where $\Lambda = \tfrac 1Q\Z\times P\Z$ with $P,Q\in\N$):
\begin{enumerate}
\item[(i)]   $Zg\in\VMO_{\rm loc}(\R^2)$.
\item[(ii)]  The Gabor system $(g,\Lambda)$ is a Riesz basis of $\calG(g,\Lambda)$.
\item[(iii)] $\pi(u,\eta)g\in\calG(g,\Lambda)$ for some $(u,\eta)\in\R^2\backslash\Lambda$.
\end{enumerate}
Due to Lemma \ref{l:rational} we may assume in (iii) that $(u,\eta)\in\Q^2\backslash\Lambda$. Our strategy is as follows: First, from (ii) and (iii) we obtain an equation of the form
$$
\bA(x,\omega) = e^{-2\pi i\eta(x+u)}D_P^{-1}\bA(x+u,\omega+\eta)\bM(x+u,\omega+\eta),
$$
which holds for a.e.\ $(x,\omega)\in\R^2$. Here, $\bA$ is the matrix function from Lemma \ref{l:riesz}, $D_P$ is a constant diagonal scaling matrix and $\bM$ is a matrix function satisfying certain periodicity properties. We then iterate this equation by successively replacing $(x,\omega)$ by $(x+u,\omega+\eta)$. As $\bA$ is quasi-periodic and $u,\eta\in\Q$, this process ends at a certain point with $\bA(x,\omega)$ on both sides and we end up with an equation
$$
\prod_{n=1}^N\bM(x+nu,\omega+n\eta) = e^{2\pi i\left(M_1x + M_2\omega\right)}I_Q.
$$
By applying the determinant on both sides, we are in the situation of Proposition \ref{p:THEPROP}, which finally implies that $(u,\eta)\in\Lambda$.

So, let us assume that (i)--(iii) are satisfied. Since the system $(g,\Lambda) = \{ \pi(\tfrac m Q,nP)g : m,n\in\Z \}$ is a Riesz basis of $\calG(g,\Lambda)$ by (ii) and $\pi(u,\eta)g\in \calG(g,\Lambda)$, there exists $(c_{m,n})_{m,n\in\Z}\in\ell^2(\Z^2)$ such that
$$
\pi(u,\eta)g = \sum_{m,n\in\Z}c_{m,n}\pi(\tfrac m Q,nP)g = \sum_{\ell=0}^{Q-1}\sum_{s,n\in\Z}c_{sQ+\ell,n}\,\pi(s+\tfrac{\ell}Q,nP)g,
$$
which converges in $L^2(\R)$. Denoting $G := Zg$, an application of the Zak transform gives (see Lemma \ref{l:zak})
\begin{align}\label{e:important}
e^{2\pi i\eta x} \, G(x-u,\omega-\eta)
&= \sum_{\ell=0}^{Q-1}\sum_{s,n\in\Z}c_{sQ+\ell,n}e^{2\pi i(nPx-s\omega)}G(x-\tfrac\ell Q,\omega) \nonumber \\
&= \sum_{\ell=0}^{Q-1}F_\ell(x,\omega) \, G(x-\tfrac\ell Q,\omega)\qquad\text{for a.e.\ $(x,\omega)\in\R^2$},
\end{align}
where $F_\ell(x,\omega) := \sum_{s,n\in\Z}c_{sQ+\ell,n}e^{2\pi i(nPx-s\omega)}$.
By definition, each $F_\ell$ is $1$-periodic in $\omega$ (i.e., $F_\ell(x,\omega+1) = F_\ell(x,\omega)$ for a.e.\ $(x,\omega)\in\R^2$) and $\frac{1}{P}$-periodic in $x$. Replacing $x$ by $x-\tfrac k P$ in \eqref{e:important}, $k=0,\ldots,P-1$, yields
$$
e^{2\pi i\eta(x-\tfrac k P)}G(x - u - \tfrac k P,\omega-\eta) = (\bA(x,\omega)F(x,\omega))_k,
$$
where $\bA$ is the matrix function from Lemma \ref{l:riesz} and $F := (F_0,\ldots,F_{Q-1})^T$. Thus, we have
\begin{equation}\label{e:also_important}
\bA(x,\omega)F(x,\omega) = e^{2\pi i\eta x}D_P\bA(x-u,\omega-\eta)e_0\qquad\text{for a.e.\ $(x,\omega)\in\R^2$},
\end{equation}
where $D_P = \diag(\exp(-2\pi i\eta\tfrac k P))_{k=0}^{P-1}$ and $e_j$ is the $(j+1)$-th standard basis vector of $\C^Q$, $j=0,\ldots,Q-1$. Note that each entry of $\bA$ is a function in $\VMOli(\R^2)$ by (i), Lemma \ref{l:affine}, and Lemma \ref{l:riesz}. The identity in \eqref{e:also_important} implies (cf.\ Lemma \ref{l:riesz})
$$
F(x,\omega) = e^{2\pi i\eta x}\left(\bA(x,\omega)^*\bA(x,\omega)\right)^{-1}\bA(x,\omega)^*D_P\bA(x-u,\omega-\eta)e_0.
$$
Hence, from the periodicity of $F_\ell$ and Corollary \ref{c:matrix} we infer that $F_\ell\in\VMO(\R^2)\cap L^\infty(\R^2)$ for $\ell = 0,\ldots,Q-1$. From \eqref{e:also_important} we also obtain
$$
\bA(x-u,\omega-\eta)e_0 = e^{-2\pi i\eta x}D_P^{-1}\bA(x,\omega)F(x,\omega).
$$
Note that $\bA(x,\omega)e_\ell = \bA(x-\tfrac\ell Q,\omega)e_0$ and $\bA(x-\tfrac\ell Q,\omega) = \bA(x,\omega)\bR(\omega)^\ell$, where $\bR(\omega)\in\C^{Q\times Q}$ is the matrix
$$
\begin{pmatrix}
0 & & & e_w\\
1 & \ddots & & \\
 & \ddots & \ddots & \\
 & & 1 & 0
\end{pmatrix},
$$
with $e_\omega := e^{-2\pi i\omega}$. Therefore,
\begin{align*}
\bA(x-u,\omega-\eta)e_\ell
&= e^{-2\pi i\eta(x-\frac\ell Q)}D_P^{-1}\bA(x-\tfrac\ell Q,\omega)F(x-\tfrac\ell Q,\omega)\\
&= e^{-2\pi i\eta(x-\frac\ell Q)}D_P^{-1}\bA(x,\omega)\bR(\omega)^\ell F(x-\tfrac\ell Q,\omega).
\end{align*}
Hence, if $\bM(x,\omega)\in\C^{Q\times Q}$ denotes the matrix with columns $e^{2\pi i\eta\frac\ell Q}\bR(\omega)^\ell F(x-\tfrac\ell Q,\omega)$, $\ell = 0,\ldots,Q-1$, we obtain
\begin{equation}\label{e:fertig}
\bA(x-u,\omega-\eta) = e^{-2\pi i\eta x}D_P^{-1}\bA(x,\omega)\bM(x,\omega).
\end{equation}
Note that $\bM$ has the same periodicity in $x$ and $\omega$ as $F$. Moreover, an easy calculation leads to
\begin{equation}\label{e:puh}
\bM(x-\tfrac 1 Q,\omega) = e^{-\frac{2\pi i}Q}\bR(\omega)^{-1}\bM(x,\omega)\bR(\omega).
\end{equation}
Now, let us iterate the relation \eqref{e:fertig}:
\begin{align*}
\bA(x,\omega)
&= e^{-2\pi i\eta(x+u)}D_P^{-1}\bA(x+u,\omega+\eta)\bM(x+u,\omega+\eta)\\
&= e^{-2\pi i\eta(2x+3u)}D_P^{-2}\bA(x+2u,\omega+2\eta)\bM(x+2u,\omega+2\eta)\bM(x+u,\omega+\eta)\\
&= \ldots\\
&= e^{2\pi i\eta\left(Nx + \frac{N(N-1)}2u\right)}D_P^{-N}\bA(x+Nu,\omega+N\eta)\prod_{n=1}^N\bM(x+nu,\omega+n\eta),
\end{align*}
where the matrix product is to be read in terms of left multiplication. As $(u,\eta)\in\Q^2$, we may choose $N$ such that $M_1 := -N\eta\in\Z$, $M_2 := -Nu\in\Z$, $Nu\eta\in 2\Z$, and $M_1/P = -N\eta/P\in\Z$. Then $D_P^{-N} = I_P$ and Lemma \ref{l:zak} yields $\bA(x+Nu,\omega+N\eta) = \bA(x-M_2,\omega-M_1) = e^{-2\pi iM_2\omega}\bA(x,\omega)$, hence
$$
\bA(x,\omega) = e^{2\pi i\left(-M_1x - M_2\omega\right)}\bA(x,\omega)\prod_{n=1}^N\bM(x+nu,\omega+n\eta).
$$
Since $\bA(x,\omega)$ has a left inverse for a.e.\ $(x,\omega)\in\R^2$ by Lemma \ref{l:riesz}, we get
$$
\prod_{n=1}^N\bM(x+nu,\omega+n\eta) = e^{2\pi i\left(M_1x + M_2\omega\right)}I_Q\qquad\text{for a.e.\ $(x,\omega)\in\R^2$}.
$$
Finally, we define the function $H := \det\bM$. Since each entry of $\bM$ is contained in $\VMOli(\R^2)$, we have $H\in\VMOli(\R^2)$. In addition, $H$ is $\frac{1}{P}$-periodic in $x$ and $1$-periodic in $\omega$. But by \eqref{e:puh}, $H$ is also $\frac{1}{Q}$-periodic in $x$. Since it satisfies
$$
\prod_{n=0}^{N-1} H(x+nu,\omega+n\eta) = e^{2\pi i\left(QM_1x + QM_2\omega\right)}\qquad\text{for a.e.\ $(x,\omega)\in\R^2$},
$$
Proposition \ref{p:THEPROP} implies that both $NP$ and $NQ$ divide $QM_1 = -NQ\eta$, and $N$ divides $QM_2 = -QNu$. The last relation gives $u\in\tfrac 1 Q\Z$. From the first two relations it follows that $P$ divides $Q\eta$ and that $\eta\in\Z$. But as $P$ and $Q$ are coprime, $P$ divides $\eta$, i.e., $\eta\in P\Z$. This is the desired contradiction.
\end{proof}

In order to extend Proposition \ref{p:separable} to arbitrary rational lattices and lattices of rational density we make use of the so-called {\em metaplectic operators}. To describe this class of operators we first mention that any matrix in $\operatorname{SL}(2,\R)$ can be expressed as a finite product of matrices of the form
\begin{equation}\label{e:generators}
\mat 0 1 {-1} 0,\quad \mat \alpha 0 0 {\alpha^{-1}},\quad\mat 1 0 \beta 1,
\qquad \alpha,\beta\in \R\backslash\{0\} .
\end{equation}
Indeed, if $S = \mat a b c d$ with $ad-bc = 1$, then
$$
S = \mat 1 0 {ca^{-1}} 1 \mat 0 1 {-1} 0\mat 1 0 {-ab} 1\mat 0 1 {-1} 0\mat {-a} 0 0 {-a^{-1}}
\quad \text{if} \;\; a \neq 0
$$
and
$$
S = \mat 1 {0} {-cd} 1\mat 0 1 {-1} 0\mat {b^{-1}} 0 0 b
\quad \text{if} \;\; a = 0 .
$$
This in particular shows that if $S\in\operatorname{SL}(2,\Q)$, then the parameters $\alpha,\beta$ can be chosen to be rational.

It is known \cite{gr} that to each matrix $S\in\operatorname{SL}(2,\R)$ there corresponds a (so-called metaplectic) unitary operator $U_S : L^2(\R)\to L^2(\R)$ such that
$$
U_S\pi(x,\omega)U_S^* = \pi(S(x,\omega)^T)
\quad \text{for all} \;\; (x,\omega)\in\R^2.
$$
The operator $U_S$ is unique up to scalar multiplication with unimodular constants. If $S,T\in\operatorname{SL}(2,\R)$, the operator $U_SU_T$ is obviously a metaplectic operator corresponding to $ST$. That is, we have $U_{ST} = U_SU_T$. As is easily seen, the three types of matrices in \eqref{e:generators}, which generate $\operatorname{SL}(2,\R)$, correspond to the metaplectic operators $\calF$ (Fourier transform), $D_\alpha$, and $C_\beta$ (defined in Proposition \ref{prop:dilations-chirps}), respectively.

\begin{proof}[Proof of Theorem \rmref{t:main-rational}]
As in the proof of Proposition \ref{p:separable}, in addition to the assumptions of Theorem \ref{t:main-rational}, we shall assume that $Zg\in\VMO_{\rm loc}(\R^2)$ and derive a contradiction. We have $\Lambda = A\Z^2$, where $A\in\operatorname{GL}(2,\Q)$. We may write $\det A = \tfrac P Q$, where $P,Q\in\Z\setminus\{0\}$ are coprime numbers. Then $B := \diag(\tfrac 1 Q,P)A^{-1}\in\operatorname{SL}(2,\Q)$ can be expressed as a finite product of matrices (\ref{e:generators}) with $\alpha,\beta\in\Q\backslash\{0\}$. Hence, the metaplectic operator $U_B$ can be written as a finite product of operators of the type $\calF$, $D_\alpha$, and $C_\beta$. Now, set
\begin{equation}\label{e:new_objects}
(u_1,\eta_1)^T := B(u,\eta)^T,\qquad g_1 := U_B g \in L^2(\R),\qquad\text{and}\qquad\Lambda_1 = B\Lambda = \tfrac{1}{Q}\Z\times P\Z.
\end{equation}
Proposition \ref{prop:dilations-chirps} implies that $Zg_1\in\VMO_{\rm loc}(\R^2)$. Also, $\pi(B\la)g_1 = U_B\pi(\la)U_B^{-1} g_1 = U_B\pi(\la)g$ for every $\la\in\Lambda$, which implies that $(g_1,\Lambda_1)$ is a Riesz basis for its closed linear span $\calG(g_1,\Lambda_1) = U_B\calG(g,\Lambda)$. Moreover, the condition $\pi(u,\eta)g\in \calG(g,\Lambda)$ immediately translates to
$$
\pi(u_1,\eta_1)g_1 = \pi(B(u,\eta)^T)U_Bg = U_B\pi(u,\eta)g
\,\in\,U_B\calG(g,\Lambda) = \calG(g_1,\Lambda_1).
$$
But as $(u_1,\eta_1)\notin\Lambda_1$ this is not possible by Proposition \ref{p:separable}.
\end{proof}

\begin{proof}[Proof of Theorem \rmref{t:main-rational_density}]
Let us first discuss the condition \eqref{eqn:uncertainty-product-pq}. Since $g\in L^2(\R)$, the condition holds for some $\alpha,\beta\in\R$ if and only if it holds for all $\alpha,\beta\in\R$. At the same time, \eqref{eqn:uncertainty-product-pq} exactly means that $g\notin H^{p/2}(\R)$ or $g\notin H^{q/2}(\R)$. Towards a contradiction, in addition to the assumptions of Theorem \ref{t:main-rational_density}, we assume that there exist some $p,q\in (1,\infty)$ with $\tfrac 1 p+\tfrac 1 q = 1$ such that the product on the left hand side of \eqref{eqn:uncertainty-product-pq} (with, e.g., $\alpha = \beta = 0$) is finite. In other words, we assume that $g\in H^{p/2}(\R)$ and $\hat g\in H^{q/2}(\R)$.

Let $\Lambda$ be an arbitrary lattice in $\R^2$ with rational density $P/Q$, where $P$ and $Q$ are coprime integers, and let $g\in L^2(\R)$ be as in Theorem \ref{t:main-rational_density}. Also, let $(u,\eta)\in\R^2\backslash\Lambda$ such that $\pi(u,\eta)g\in \calG(g,\Lambda)$. As in the proof of Theorem \rmref{t:main-rational}, we choose a matrix $B\in\operatorname{SL}(2,\R)$ (here, $B$ is allowed to have irrational entries) such that $B\Lambda = \Lambda_1 := \frac{1}{Q}\Z\times P\Z$. Define $(u_1,\eta_1)$, $g_1$, and $\Lambda_1$ as in \eqref{e:new_objects}. Then $g_1\in L^2(\R)$, $(g_1,\Lambda_1)$ is a Riesz basis for its closed linear span, and $\pi(u_1,\eta_1)g_1\in\calG(g_1,\Lambda_1)$. Hence, by Proposition \ref{p:separable} it suffices to show that $Zg_1\in\VMO_{\rm loc}(\R^2)$.

To simplify notations, for $p\in (1,\infty)$ let $\H^{p}(\R)$ be the space of functions $f\in H^{p/2}(\R)$ whose Fourier transform $\hat f$ is contained in $H^{q/2}(\R)$, where $\tfrac 1 p + \tfrac 1 q = 1$. Also, let $\H := \bigcup_{p\in (1,\infty)}\H^p(\R)$. In what follows, we prove that $U\H\subset\H$ for any metaplectic operator $U$. It then follows that $g_1 = U_Bg\in\H$. And since, by \cite{g}, we have $Z\H\subset\VMO_{\rm loc}(\R^2)$, we obtain $Zg_1\in\VMO_{\rm loc}(\R^2)$, which was our aim.

Since every metaplectic operator is a finite product of the Fourier transform $\calF$, dilations, and chirp muliplication, we only have to prove that $\calF\H\subset\H$, $D_\alpha\H\subset\H$ for each $\alpha\in\R\setminus\{0\}$, and $C_1\H\subset\H$ (cf.\ \eqref{e:easy_chirp}). Using the representation $H^s(\R) = \{f\in L^2(\R) : \omega^{2s}\wh f\,^2\in L^1(\R)\}$, the first two claims are almost immediate. We will now prove that $C_1\H^p(\R)\subset\H^p(\R)$ for $p\in (1,\infty)$. So, let $g\in\H^p(\R)$. Then $\hat g\in H^{q/2}(\R)$ implies $x^q g^2\in L^1(\R)$ and hence $\wh{C_1g}\in H^{q/2}(\R)$. In order to show that $C_1g\in H^{p/2}(\R)$, we make use of the following representation of fractional Sobolev spaces (see, e.g., \cite{npv}):
$$
H^s(\R) = \left\{f\in L^2(\R) : \frac{f(x)-f(y)}{(x-y)^{\tfrac 1 2 + s}}\in L^2(\R^2)\right\}
$$
for $s\in (0,1)$ and
$$
H^s(\R) = \left\{f\in H^m(\R) : f^{(m)}\in H^\sigma(\R)\right\}
$$
for $s = m+\sigma$ with $m\in\N$ and $\sigma\in [0,1)$. Let $c(x) := e^{2\pi ix^2}$, $x\in\R$. Then $C_1g(x)=c(x)g(x)$ and it is clear that $C_1g\in H^{p/2}(\R)$ if $p/2\in\N$. We only prove the claim here for $1 < p < 2$. The rest is then straightforward. We have (setting $s=p/2$)
$$
\frac{C_1g(x)-C_1g(y)}{(x-y)^{\frac{p+1}2}} = c(x)\frac{g(x)-g(y)}{(x-y)^{\frac{p+1}2}} + \frac{c(x)-c(y)}{(x-y)^{\frac{p+1}2}}g(y).
$$
Hence, as $g\in H^{p/2}(\R)$, it is left to show that the second summand is in $L^2(\R^2)$. To this end, fix $y\in\R$ and observe that
$$
\int_{|x-y|>1}\frac{|c(x)-c(y)|^2}{|x-y|^{p+1}}\,dx\,\le\,c_1
$$
and
$$
\int_{|x-y|\le 1}\frac{|c(x)-c(y)|^2}{|x-y|^{p+1}}\,dx\,\le\,c_2(1+y^2),
$$
where $c_1,c_2 > 0$ only depend on $p$. Hence,
$$
\int_{\R}\frac{|c(x)-c(y)|^2}{|x-y|^{p+1}}\,dx\,\le\,c(1+y^2)
$$
for each $y\in\R$. Now, as $\hat g\in H^{q/2}(\R)$, we have $(1+y^2)^{q/2}|g(y)|^2\in L^1(\R)$, and $p < 2$ implies $q/2>1$. Thus, also $(1+y^2)|g(y)|^2\in L^1(\R)$, and the proof is complete.
\end{proof}

\begin{prob}\label{prob:open}
As already mentioned, we have restricted ourselves to rational lattices in Theorem \ref{t:main-rational} while Theorem \ref{t:main-rational_density} considers a broader class of lattices, namely the lattices of rational density. The main reason for this is that we could not prove whether the set $\{ g \in L^2(\R) : Zg \in\VMO_{\rm loc}(\R^2) \}$ is invariant under irrational dilations. If this were true, Theorem \ref{t:main-rational} would hold not only for rational lattices but for arbitrary lattices of rational density.
We leave the following as open problems:
\begin{enumerate}
\item Is it true that $Zg\in\VMO_{\rm loc}(\R^2)$ for $g\in L^2(\R)$ implies $Z(D_\alpha g)\in\VMO_{\rm loc}(\R^2)$ for every $\alpha\in\R\backslash\Q$?
\item Is there a good description of the space consisting of the functions $g\in L^2(\R)$ that satisfy $Zg\in\VMO_{\rm loc}(\R^2)$?
\end{enumerate}
\end{prob}

\bigskip\noindent
{\bf Acknowledgements.} D.G.~Lee and G.E.~Pfander acknowledge support by the DFG Grants PF 450/6-1 and PF 450/9-1. The authors would like to thank S. Zubin Gautam, Jeffrey A. Hogan, and David F. Walnut for valuable discussions. They also thank the referees for having read the manuscript carefully and their various valuable comments.

\vspace{1cm}
\section*{Author affiliations}
\end{document}